\documentclass[a4paper,12pt]{amsart}
\usepackage{amssymb}

\numberwithin{equation}{section}

\newtheorem{thm}[equation]{Theorem}
\newtheorem{prop}[equation]{Proposition}
\newtheorem{cor}[equation]{Corollary}
\theoremstyle{definition}
\newtheorem{defn}[equation]{Definition}
\newtheorem{ex}[equation]{Example}
\newtheorem{remark}[equation]{Remark}

\DeclareMathOperator{\ch}{\mathrm{ch}}
\DeclareMathOperator{\td}{\mathrm{td}}
\DeclareMathOperator{\so}{\mathrm{SO}}
\DeclareMathOperator{\Sp}{\mathrm{Sp}}
\DeclareMathOperator{\ha}{\mathbb{H}_{\beta}}

\begin{document}

\title{Indices of quaternionic complexes}
\author{Old\v{r}ich Sp\'{a}\v{c}il}
\address{Department of Mathematics and Statistics, Masaryk University, Kotl\'{a}\v{r}sk\'{a} 2, 611 37 Brno, Czech Republic}
\email{151393@mail.muni.cz}

\begin{abstract}
Methods of parabolic geometries have been recently used to construct a class of elliptic complexes on quaternionic manifolds, the Salamon's complex being the simplest case. The purpose of this paper is to describe an algorithm how to compute their analytical indices in terms of characteristic classes. Using this, we are able to derive some topological obstructions to existence of quaternionic structures on manifolds.
\end{abstract}


% 53C15 General geometric structures on manifolds (almost complex, almost product structures, etc.)
% 53C26 Hyper-Kahler and quaternionic Kahler geometry, "special" geometry
% 57R20 Characteristic classes and numbers
% 58J10 Differential complexes, elliptic complexes
%\end{keyword}
\maketitle

%%%%%%%%%%%%%%%%%%%%%%%%%%%%%%%%%%%%%%%%%%%%%%%%%%%%%%%%%%%%%%%%%%%%%%%%%%%%%%%%%%%%%%%%%%%%%%%%%%%%%%%%%%%%%%%%%%%%%%%%%%%%%%%%%%%%%%%%%%%%%%%%%%%%%%%%%%%%%%%%%%%%%%%%%%%%%%%%%%%%%%%%%%%%%%%%%%%%%%%%%%%%%%%%%%%%%%%%%%%%%%%%%%%%%%%%%%%%%%%%%%%%%%%%%%%%%%%%%%%%%%%%%%%%%%%%%%%%%%%%%%%%%%%%%%%%%%%%%%%%%%%%%%%%%%%%%%%%%%%%%%%%%%%%%%%%%%%%%%%%%%%%%%%%%%%%%%%%%%%%

\section{Introduction}

Quaternionic geometry has become a classical part of modern differential geometry. The purpose of the present paper is to contribute to the study of differential operators which are intrinsic to quaternionic manifolds and to the study of topology of such manifolds. The connection is given via the analytical index of elliptic complexes of differential operators.

A large class of elliptic complexes on quaternionic manifolds has been recently constructed in a much more general setting in \cite{SUB}. The point is that quaternionic geometry is an example of the so-called \emph{parabolic geometries} (see \cite{PAR}) and there is a well-developed machinery to construct invariant differential operators on such geometries. These operators fit naturally into sequences, the \emph{curved Bernstein-Gelfand-Gelfand sequences}, see \cite{BGG}. Although a BGG-sequence does not form a complex in general, it was shown in \cite{SUB} that under some torsion-freeness like conditions it contains many subcomplexes. For quaternionic manifolds some of them turned out to be elliptic.

Our aim was to compute the analytical indices of these elliptic complexes on compact manifolds. Some results in this direction were obtained in \cite{Bas} for another class of elliptic complexes in the case of hyperk\"{a}hler manifolds and in \cite{LBS} for positive quaternion K\"{a}hler manifolds. However, it seems that there is no such result for general quaternionic manifolds.

Techniques of both these articles rely heavily on a correspondence between the quaternionic manifold and its twistor space, which is a complex manifold and thus one can apply the Hirzebruch-Riemann-Roch theorem. Our approach is more elementary and computational in nature, the main tool is a version of the Atiyah-Singer index theorem for elliptic complexes associated to a $G$-structure. The theorem provides us with a topological formula involving several characteristic classes but to simplify this we need some "common language." Luckily, this was described in \cite{QS}, where a quaternionic structure is viewed as an action of a bundle of quaternion algebras. Using this and the method of Borel and Hirzebruch (see \cite{BH}) for computation of Chern classes, we have developed an algorithm how to obtain a formula for the analytical index for each given complex and in each given dimension.

The resulting formulas are given in terms of the Pontryagin classes of the manifold and one another cohomology class which is related to the quaternionic structure of the manifold. By taking integral linear combinations of the index formulas it is possible to eliminate this class and derive some integrality conditions on the existence of a quaternionic structure on a compact manifold (and not only of an almost quaternionic structure), for an example see Corollary \ref{cond}.   

%%%%%%%%%%%%%%%%%%%%%%%%%%%%%%%%%%%%%%%%%%%%%%%%%%%%%%%%%%%%%%%%%%%%%%%%%%%%%%%%%%%%%%%%%%%%%%%%%%%%%%%%%%%%%%%%%%%%%%%%%%%%%%%%%%%%%%%%%%%%%%%%%%%%%%%%%%%%%%%%%%%%%%%%%%%%%%%%%%%%%%%%%%%%%%%%%%%%%%%%%%%%%%%%%%%%%%%%%%%%%%%%%%%%%%%%%%%%%%%%%%%%%%%%%%%%%%%%%%%%%%%%%%%%%%%%%%%%%%%%%%%%%%%%%%%%%%%%%%%%%%%%%%%%%%%%%%%%%%%%%%%%%%%%%%%%%%%%%%%%%%%%%%%%%%%%%%%%%%%%

\section{Quaternionic manifolds and quaternionic complexes}\label{1sec}

Let us first recall the classical definition of quaternionic manifolds via $G$-structures. View $\mathbb{R}^{4m}$ as the space $\mathbb{H}^{m}$ of $m$-tuples of quaternions. Then the group $\Sp(1)$ of unit quaternions acts on it by $a\cdot v=v\bar{a}$ and $\mathrm{GL}(m,\mathbb{H})$ acts by left matrix multiplication. These actions induce injections $\Sp(1) \rightarrow \mathrm{GL}(4m,\mathbb{R})$ and $\mathrm{GL}(m,\mathbb{H}) \rightarrow\mathrm{GL}(4m,\mathbb{R})$. We define the group $\Sp(1)\mathrm{GL}(m,\mathbb{H})$ as the product of these two groups in $\mathrm{GL}(4m,\mathbb{R})$. More abstractly, $\Sp(1)\mathrm{GL}(m,\mathbb{H})$ is isomorphic to the quotient $\Sp(1)\times_{\mathbb{Z}_2}\mathrm{GL}(m,\mathbb{H})$, where a pair $(a,A)\in \Sp(1)\times\mathrm{GL}(m,\mathbb{H})$ represents the linear mapping $v\mapsto Av\bar{a}$. In the sequel we write simply $G$ instead of $\Sp(1)\mathrm{GL}(m,\mathbb{H})$.

\begin{defn}
A $4m$-dimensional manifold $M$, $m\geq 2$, is called \emph{almost quaternionic} if it has a $G$-structure, i.e. there is a principal $G$-bundle $\mathcal{P}$ and an isomorphism of vector bundles $\mathcal{P}\times_{G} \mathbb{R}^{4m} \cong TM$. An almost quaternionic manifold is called \emph{quaternionic} if the $G$-structure $\mathcal{P}$ admits a torsion-free connection. 
\end{defn}

Denote by $G_1$ the group $\Sp(1)\times\mathrm{GL}(m,\mathbb{H})$, which is the double cover of $G$. The principal $G$-bundle $\mathcal{P}$ can be locally lifted to a principal $G_1$-bundle $\mathcal{P}_1$. Then if $\rho\colon G_1 \to \mathrm{GL}(\mathbb{V})$ is a representation of $G_1$, we can locally construct the associated vector bundle $V = \mathcal{P}_1\times_{G_1} \mathbb{V}$. This vector bundle exists globally if either the lift can be done globally or the action $\rho$ factors through $G$. 

Let $\mathbb{E}$ and $\mathbb{F}$ denote the standard complex $\Sp(1)$-module and $\mathrm{GL}(m,\mathbb{H})$-module, respectively. These can be also viewed as $G_{1}$-modules by composing the corresponding representations with the projections onto the two factors. Then the $G_{1}$-module $\mathbb{E}\otimes_{\mathbb{C}}\mathbb{F}$ descends to a $G$-module and one can show that we have the following isomorphisms of vector bundles  
\begin{equation}\label{EF} T^{*}M\otimes_{\mathbb{R}}\mathbb{C} \cong \mathcal{P}\times_{G} (\mathbb{E}\otimes_{\mathbb{C}}\mathbb{F})  \cong \mathcal{P}_{1}\times_{G_1} (\mathbb{E}\otimes_{\mathbb{C}}\mathbb{F}) \cong E\otimes_{\mathbb{C}}F.\end{equation}
However, note that the bundles $E$ and $F$ may not exist globally.

An easy consequence of the above decomposition is that there is a natural subcomplex of the de\,Rham complex of complex-valued differential forms. Indeed, decomposing the $G_1$-module $\Lambda^{j}(\mathbb{E}\otimes \mathbb{F})$ into the sum of irreducible $G_{1}$-modules we can find a summand $\mathbb{A}^{j}$ isomorphic to $S^{j}\mathbb{E}\otimes \Lambda^{j}\mathbb{F}$. This is in fact a $G$-module because the action of $G_1$ on $S^{j}\mathbb{E}\otimes \Lambda^{j}\mathbb{F}$ factors through $G$. Therefore, in view of \eqref{EF} this implies that there is a vector subbundle $A^{j}\subseteq \Lambda^{j}(T^{*}M\otimes\mathbb{C})$ such that
\begin{equation}\label{Abundle} A^{j} = \mathcal{P}\times_{G} \mathbb{A}^{j} \cong S^{j}E\otimes\Lambda^{j}F.\end{equation}
Let $d$ denote the exterior derivative and $p_{j}\colon \Lambda^{j}(T^{*}M\otimes\mathbb{C}) \to A^{j}$ the projection. If we put $d_{j}=p_{j}\circ d$, we obtain the following sequence of differential operators
\begin{equation}\label{complex} 0 \rightarrow \Gamma A^0 \xrightarrow{d_{1}} \Gamma A^1 \xrightarrow{d_{2}} \Gamma A^{2} \xrightarrow{d_3} \cdots \xrightarrow{d_{2m}} \Gamma A^{2m} \rightarrow 0.\end{equation}
This sequence is closely related to the (almost) quaternionic structure of the manifold.

\begin{prop}[Salamon]
An almost quaternionic manifold $M$ is quaternionic if and only if the sequence (\ref{complex}) is a complex. If this is the case, then the complex is elliptic.
\end{prop}
\begin{proof}
The proof can be found in \cite{Sal}, Theorem 4.1. 
\end{proof}

In the rest of the paper, the elliptic complex \eqref{complex} will be called the \emph{Salamon's complex} because it was first constructed by S.\,Salamon in \cite{Sal} together with several other elliptic complexes, which may not exist globally. Another class of elliptic complexes on quaternionic manifolds was described by R.\,Baston in \cite{Bas}. Recently, A.\,\v{C}ap and V.\,Sou\v{c}ek applied the theory of parabolic geometries to construct a wide class of complexes of differential operators as subcomplexes in the so-called curved BGG-sequences, see \cite{SUB}, and many of them turned out to be elliptic. We will describe these complexes in a convenient way without giving their explicit construction. This can be found in \cite{SUB}, the theory of BGG-sequences in \cite{BGG} and the general theory of parabolic geometries in \cite{PAR}.      

According to \cite{Sal} the complex representation theories of the Lie groups $\mathrm{GL}(m,\mathbb{H})\subset \mathrm{GL}(2m,\mathbb{C})$ and $\mathrm{U}(2m)$ are equivalent. In particular, irreducible complex representations of $\mathrm{GL}(m,\mathbb{H})$ may be identified with highest weights for $\mathrm{U}(2m)$. For each $k\geq 0$ put \begin{equation}\label{w}\mathbb{W}_{k}^{j} = S^{j+k}\mathbb{E}\otimes (\Lambda^{j}\mathbb{F}\otimes S^{k}\mathbb{F}^{*})_{0} \quad \text{for } j<2m, \quad \mathbb{W}_{k}^{2m} = S^{2(m+k)}\mathbb{E}\otimes \Lambda^{2m}\mathbb{F},\end{equation} where the zero subscript denotes the irreducible component of the tensor product with the highest weight being the sum of the highest weights of the respective factors. Note that the $\mathbb{W}_{k}^{j}$ are $G_{1}$ as well as $G$-modules.

\begin{prop}[\v{C}ap, Sou\v{c}ek]\label{quatcomplex}
Let $M$ be a $4m$-dimensional quaternionic manifold with a $G$-structure $\mathcal{P}$. Denote by $W_{k}^{j}$ the associated vector bundle $\mathcal{P}\times_{G} \mathbb{W}_{k}^{j}$. Then for each $k\geq 0$ there is an elliptic complex $$D_{k}\colon \quad 0\to \Gamma W_{k}^{0} \xrightarrow{D} \Gamma W_{k}^{1} \xrightarrow{D} \cdots \xrightarrow{D} \Gamma W_{k}^{2m} \to 0.$$ 
\end{prop}
\begin{proof}
For the proof see \cite{SUB}. Let us remark, however, that the existence of a compatible torsion-free connection is essential in the construction. Therefore, this result does not hold for manifolds which are only almost quaternionic.
\end{proof}

The differential operators $D$ are explicitly constructed in \cite{BGG} in a much more general setting of parabolic geometries. In the case of quaternionic manifolds they are shown to be \emph{strongly invariant}, which implies that the symbol of $D$ is induced by a $G$-equivariant polynomial map $\mathbb{R}^{4m} \to L(\mathbb{W}_{k}^{j}, \mathbb{W}_{k}^{j+1})$. This will be cleared out in the next section.

The goal of the paper is to compute the analytical indices of the elliptic complexes $D_{k}$. Note that by setting $k=0$ we obtain precisely the Salamon's complex $\eqref{complex}$. Moreover, the complex $D_{1}$ may be naturally interpreted as a deformation complex for quaternionic structures, see \cite{DEF}.

%%%%%%%%%%%%%%%%%%%%%%%%%%%%%%%%%%%%%%%%%%%%%%%%%%%%%%%%%%%%%%%%%%%%%%%%%%%%%%%%%%%%%%%%%%%%%%%%%%%%%%%%%%%%%%%%%%%%%%%%%%%%%%%%%%%%%%%%%%%%%%%%%%%%%%%%%%%%%%%%%%%%%%%%%%%%%%%%%%%%%%%%%%%%%%%%%%%%%%%%%%%%%%%%%%%%%%%%%%%%%%%%%%%%%%%%%%%%%%%%%%%%%%%%%%%%%%%%%%%%%%%%%%%%%%%%%%%%%%%%%%%%%%%%%%%%%%%%%%%%%%%%%%%%%%%%%%%%%%%%%%%%%%%%%%%%%%%%%%%%%%%%%%%%%%%%%%%%%%%%

\section{Atiyah-Singer index formula}
In this section we present a version of the Atiyah-Singer index theorem for elliptic complexes associated to $G$-structures, which is our main tool for computation of the indices. At the end we apply this to the Salamon's complex on manifolds with a $\mathrm{GL}(m,\mathbb{H})$-structure. 

Let $G$ be a Lie group and let $M$ be a compact oriented manifold with a $G$-structure $\mathcal{P}$, i.e. there is a principal $G$-bundle $\mathcal{P} \to M$ and a real oriented $G$-module $\mathbb{V}$ such that $TM \cong \mathcal{P}\times_{G} \mathbb{V}$. Let $\mathbb{E}^{j}$, $0\leq j\leq r$, be complex $G$-modules and put $E^{j} = \mathcal{P}\times_{G}\mathbb{E}^{j}$. Denote by $D$ an elliptic complex $$0 \rightarrow \Gamma E^0 \xrightarrow{d_0} \Gamma E^1 \xrightarrow{d_1} \cdots \xrightarrow{d_{r-1}} \Gamma E^{r} \rightarrow 0$$ of differential operators on $M$. Suppose that $\varphi_{j}\colon \mathbb{V}^{*} \to L(\mathbb{E}^{j},\mathbb{E}^{j+1})$ is a $G$-equivariant polynomial map such that for all $v\in \mathbb{V}^{*}, v\neq 0$, the sequence
\begin{equation}\label{univ}0\rightarrow \mathbb{E}^{0} \xrightarrow{\varphi_{0}(v)} \mathbb{E}^{1} \xrightarrow{\varphi_{1}(v)} \cdots \xrightarrow{\varphi_{r-1}(v)} \mathbb{E}^{r} \rightarrow 0\end{equation}
is exact. If the symbol $\sigma_{D}$ of $D$ is induced via the isomorphisms $T^{*}M\cong \mathcal{P}\times_{G} \mathbb{V}^{*}$ and $E^{j}\cong \mathcal{P}\times_{G} \mathbb{E}^{j}$ from the sequence \eqref{univ}, then we say that $\sigma_{D}$ is \emph{associated to the $G$-structure $\mathcal{P}$}. More explicitly, the symbol of the operator $d_{j}$ is a fibrewise polynomial bundle map $$\sigma_{d_{j}}\colon T^{*}M \cong \mathcal{P}\times_{G} \mathbb{V}^{*} \to L(E^{j}, E^{j+1}) \cong \mathcal{P}\times_{G} L(\mathbb{E}^{j}, \mathbb{E}^{j+1})$$ and we require that $\sigma_{d_{j}}([p,v]) = [p, \varphi_{j}(v)]$.

For example, the symbol of the de\,Rham complex of complex-valued differential forms at $(x,v)\in T^{*}M$ is given by the exterior product on $v$ and this is clearly associated via the mappings $\varphi_{j}(v) = v\wedge-$ to a $\so(m)$-structure of $M$ induced by a Riemannian metric and an orientation. 

As was already noted at the end of the previous section, the differential operators forming the quaternionic complexes from Proposition \ref{quatcomplex} are strongly invariant. This implies (see \cite{BGG}) that their symbols are associated to the $\Sp(1)\mathrm{GL}(m,\mathbb{H})$-structure of our quaternionic manifold $M$.

Before we state the version of the Atiyah-Singer index formula we will need, let us recall some basic facts on classifying spaces. Let $G$ be a compact Lie group and let $EG\to BG$ denote the universal principal $G$-bundle over the classifying space $BG$. If $\mathcal{P} \to M$ is a principal $G$-bundle, then there is up to homotopy a unique map $f\colon M \to BG$ such that $\mathcal{P}\cong f^{*}EG$, the pullback bundle. In cohomology, $f$ induces a ring homomorphism $f^{**}\colon H^{**}(BG;\mathbb{Q}) \to H^{**}(M;\mathbb{Q})$, where $H^{**}(-;\mathbb{Q}) = \prod_{k\geq 0}H^{k}(-;\mathbb{Q})$. Finally, it is shown in \cite{Bor} that $H^{**}(BG;\mathbb{Q})$ is a ring of formal power series in several indeterminates with rational coefficients, hence an integral domain. In the following, $\ch$ denotes the Chern character and $\td$ the Todd class.

\begin{prop}[Atiyah, Singer]\label{formula}
Let $M$ be a compact oriented manifold of dimension $2m$, $G$ a compact Lie group and $\rho\colon G\to \mathrm{SO}(2m)$ a Lie group homomorphism. Assume that $M$ has a $G$-structure $\mathcal{P}$, i.e. $TM$ is associated to $\mathcal{P}$ via $\rho$. Let $\mathbb{E}^{j}$, $0\leq j \leq r$, be complex $G$-modules and let $E^{j}$ be the corresponding associated vector bundles. Suppose that $$0 \rightarrow \Gamma E^0 \xrightarrow{d_0} \Gamma E^1 \xrightarrow{d_1} \cdots \xrightarrow{d_{r-1}} \Gamma E^{r} \rightarrow 0$$ is an elliptic complex with its symbol associated to the $G$-structure $\mathcal{P}$. Let $f\colon M \to BG$ be the classifying map for the bundle $\mathcal{P}$. Put $\widetilde{E}^{j}=EG\times_{G} \mathbb{E}^{j}$ and $\widetilde{V}=EG\times_{\rho} \mathbb{R}^{2m}$. If the Euler class  $e(\widetilde{V})$ is nonzero, then it divides $\sum (-1)^{j}\mathrm{ch}\,\widetilde{E}_{j}\in H^{**}(BG;\mathbb{Q})$ and the index of the above complex is given by $$(-1)^{m}\left\{f^{**}\left(\frac{\sum_{j=0}^{r} (-1)^{j}\mathrm{ch}\,\widetilde{E}^{j}}{e(\widetilde{V})} \right)\cdot\mathrm{td}(TM\otimes\mathbb{C})\right\}[M].$$
\end{prop}
\begin{proof}
The proof can be found in \cite{AS3}.
\end{proof}

The proposition will be the main tool for our computation of the indices in question. However, we first need to reduce the structure group of the principal $\Sp(1)\mathrm{GL}(m,\mathbb{H})$-bundle $\mathcal{P}$ of the quaternionic manifold $M$ to a compact subgroup. But this can be easily done by introducing a Riemannian metric on $M$ because $(\Sp(1)\mathrm{GL}(m,\mathbb{H})) \cap \mathrm{O}(4m) = \Sp(1)\Sp(m)$ is a compact subgroup of $\mathrm{SO}(4m)$. Moreover, it follows that a quaternionic manifold is always orientable and thus Proposition \ref{formula} may be applied if the universal Euler class is not zero. This will be shown in the next section, where we study certain characteristic classes useful for the actual calculation of the index formula.

\begin{remark}[Chern classes]\label{borhirz}
For the computation of Chern classes and Chern characters we use the approach of Borel and Hirzebruch from \cite{BH} which relates characteristic classes of associated vector bundles to weights of the corresponding representations. Let us outline this in a few lines.

Let $\lambda\colon G \to \mathrm{U}(\mathbb{V}) \cong \mathrm{U}(n)$ be a complex representation of a compact Lie group $G$. Let $S\subseteq G$ be a maximal torus of $G$ such that $S$ is mapped via $\lambda$ to the maximal torus of all diagonal matrices in $\mathrm{U}(n)$. Suppose that $x_{1}, x_{2}, \ldots, x_{n}$ are the weights of $\lambda$ with respect to $S$. Having a principal $G$-bundle $\mathcal{P}\to M$, consider the associated vector bundle $V=\mathcal{P}\times_{G}\mathbb{V}$. Then the total Chern class, the Chern character and the Todd class of $V$ may be formally written as $$c(V) = 1 + c_{1}(V) + \ldots + c_{n}(V) = \prod_{j=1}^{m}(1+y_{j}), \quad \ch(V) = \sum_{j=1}^{n} \mathrm{e}^{y_{j}},\quad \td(V) = \prod_{j=1}^{n}\frac{y_{j}}{1-\mathrm{e}^{-y_{j}}}.$$ for some two-dimensional integral cohomology classes $y_{j}$ derived from the weights $x_{j}$.

Similarly, the weights of the $G$-module $\Lambda^{k}\mathbb{V}$ are the sums $x_{j_1}+x_{j_2}+\ldots+x_{j_{k}}$, where $1\leq j_{1}< j_{2}<\ldots j_{k} \leq n,$ and thus we can again write the Chern classes and the Chern character $\Lambda^{j}(V)$ in terms of the $y_{j}$. The following formula for the Chern character of the formal polynomial $\Lambda_{t}(V)= \sum_{k=0}^{n} t^{k}\Lambda^{k}V$ will be useful $$\ch(\Lambda_{t}(V)) = \sum_{k=0}^{n} t^{k}\left(\sum_{1\leq j_1 <\ldots < j_{k}\leq n} \mathrm{e}^{y_{j_1}+\ldots+y_{j_{k}}} \right) = \prod_{j=1}^{n}(1+t\mathrm{e}^{y_{j}}).$$
\end{remark}
  
Before we continue to study quaternionic structures from a more topological viewpoint, let us apply Proposition \ref{formula} to compute the index of the Salamon's complex \eqref{complex} in the special case of manifolds admitting a $\mathrm{GL}(m,\mathbb{H})$-structure.

\begin{ex}\label{hyper}
Let $M$ be a compact $4m$-dimensional manifold with a $\mathrm{GL}(m,\mathbb{H})$-structure admitting a torsion-free connection. Because $\mathrm{GL}(m,\mathbb{H})$ can be identified with a subgroup of $\Sp(1)\mathrm{GL}(m,\mathbb{H})$, the manifold $M$ is quaternionic and so the Salamon's complex has sense. However, in this case both the bundles $E$ and $F$ exist globally and $E$ is trivial. Moreover, the cotangent bundle $T^{*}M$ is isomorphic to the complex vector bundle $F$ up to orientation -- for $m$ even the orientations coincide and for $m$ odd they are opposite. Indeed, $T^{*}M$ is oriented as a quaternionic vector bundle while $F$ as a complex vector bundle. The $\mathrm{GL}(m,\mathbb{H})$-modules $\mathbb{A}^{j}$ inducing the vector bundles in the Salamon's complex now look like $\mathbb{C}^{j+1}\otimes\Lambda^{j} \mathbb{F}$.
    
By introducing a Riemannian metric on $M$ we may reduce the $\mathrm{GL}(m,\mathbb{H})$-structure of $M$ to a $\Sp(m)$-structure $\mathcal{P}$ (not necessarily admitting a torsion-free connection) and then apply Proposition \ref{formula} with $\rho\colon \Sp(m)\hookrightarrow \so(4m)$ being the standard inclusion described in Section \ref{1sec} and $\mathbb{E}^{j}=\mathbb{A}^{j}$. The Euler class of the universal vector bundle  $\widetilde{V}=E\mathrm{Sp}(m)\times_{\rho}\mathbb{R}^{4m}$ is one of the generators of the cohomology ring of $B\mathrm{Sp}(m)$ and thus it is nonzero.

Consider the universal vector bundles $\widetilde{F}=E\mathrm{Sp}(m)\times_{\Sp(m)}\mathbb{F}$ and $\widetilde{A}^{j}=E\mathrm{Sp}(m)\times_{\Sp(m)}\mathbb{A}^{j}$. Then $\widetilde{V}\cong \widetilde{F}$ up to orientation and $\widetilde{A}^{j} \cong \mathbb{C}^{j+1}\otimes\Lambda^{j}\widetilde{F}$. Because $\widetilde{F}$ comes from a quaternionic vector bundle, we have  $\overline{\widetilde{F}} \cong \widetilde{F}$ and altogether this gives $$e(\widetilde{V})=(-1)^{m}c_{2m}(\widetilde{F}), \quad \td(\widetilde{V}\otimes\mathbb{C}) = \td(\widetilde{F}\oplus\overline{\widetilde{F}})=\td(\widetilde{F})^{2}.$$ Finally, let $f\colon M\to B\mathrm{Sp}(m)$ be the classifying map for $\mathcal{P}$. Then according to Proposition \ref{formula} the index of the Salamon's complex is given by
\begin{equation}\label{indhk}\mathrm{ind} = f^{**}\left(\frac{\sum_{j=0}^{2m} (-1)^{j}(j+1)\mathrm{ch}\,\Lambda^{j}\widetilde{F}}{(-1)^{m}c_{2m}(\widetilde{F})}\cdot \mathrm{td}(\widetilde{F})^2\right)[M].\end{equation}

To simplify this formula we compute the Chern classes of the complex vector bundle $\widetilde{F}$ as described in Remark \ref{borhirz}. Let $S$ be the maximal torus of $\Sp(m)$ consisting of all diagonal matrices with entries $\exp(2\pi\mathrm{i}x_{j})$, where $x_{j}\in\mathbb{R}$. Then the weights of the $\Sp(m)$-module $\mathbb{F}$ are $\pm x_{j}, 1\leq j\leq m$, viewed as linear forms on the Lie algebra $\mathfrak{s}$. It follows that the total Chern class and the Todd class of $\widetilde{F}$ may be written in the form $$ c(\widetilde{F}) = 1+c_{1}(\widetilde{F}) + \ldots +c_{2m}(\widetilde{F}) = \prod_{j=1}^{m}(1+y_{j})(1-y_{j}), \quad \td(\widetilde{F}) = \prod_{j=1}^{m} \frac{y_{j}(-y_{j})}{(1-\mathrm{e}^{-y_{j}})(1-\mathrm{e}^{y_{j}})}.$$ In particular, the last Chern class equals $c_{2m}(\widetilde{F}) = \prod_{j=1}^{m}y_{j}(-y_{j})$. 

The numerator of the fraction in \eqref{indhk} may be simplified as follows. Instead of $(-1)^{j}$ write $t^{j}$ and recall the formula for $\ch(\Lambda_{t}(V))$ from Remark \ref{borhirz}. Then
\begin{equation*}
\begin{split}
&\sum_{j=0}^{2m}\,(j+1)t^{j}\ch(\Lambda^{j}\widetilde{F}) = \frac{\mathrm{d}}{\mathrm{d}t}\left( \sum_{j=0}^{2m} t^{j+1}\ch (\Lambda^{j}\widetilde{F}) \right) = \\ &= \frac{\mathrm{d}}{\mathrm{d}t}\left( t\cdot \ch(\Lambda_{t}(\widetilde{F}) \right) = \frac{\mathrm{d}}{\mathrm{d}t}\left( t\prod_{j=1}^{m}(1+t\mathrm{e}^{y_{j}})(1+t\mathrm{e}^{-y_{j}})\right) = \\ &=\prod_{j=1}^{m}(1+t\mathrm{e}^{y_{j}})(1+t\mathrm{e}^{-y_{j}}) + t\sum_{j=1}^{m} (\mathrm{e}^{y_{j}}(1+t\mathrm{e}^{-y_{j}})+\mathrm{e}^{-y_{j}}(1+t\mathrm{e}^{y_{j}}))\prod_{\substack{k=1\\k\neq j}}^{m}(1+t\mathrm{e}^{y_{k}})(1+t\mathrm{e}^{-y_{k}}).
\end{split}
\end{equation*}      
Substituting $t=-1$ and collecting the terms we end up with $$\sum_{j=0}^{2m} (-1)^{j}(j+1)\mathrm{ch}\,\Lambda^{j}\widetilde{F} = (m+1)\prod_{j=1}^{m}(1-\mathrm{e}^{y_{j}})(1-\mathrm{e}^{-y_{j}})$$
The interior of the bracket in \eqref{indhk} now reads as $$\frac{(m+1)\prod_{j=1}^{m}(1-\mathrm{e}^{y_{j}})(1-\mathrm{e}^{-y_{j}})}{(-1)^{m}\prod_{j=1}^{m}y_{j}(-y_{j})} \cdot \left( \prod_{j=1}^{m} \frac{y_{j}(-y_{j})}{(1-\mathrm{e}^{-y_{j}})(1-\mathrm{e}^{y_{j}})} \right)^{2} = (-1)^{m}(m+1)\td(\widetilde{F}).$$    

Applying the map $f^{**}$ and evaluating on the fundamental class of $M$ we obtain the desired index. The complex tangent bundle $T^{c}M$ of $M$ is isomorphic to the vector bundle $F^{*}\cong F = \mathcal{P}\times_{\Sp(m)}\mathbb{F}$. We have thus proved the following theorem, which is our first partial result on indices of quaternionic complexes. 

\begin{thm}\label{hyperindex}
Let $M$ be a compact manifold with a $\mathrm{GL}(m,\mathbb{H})$-structure admitting a torsion-free connection. Then the index of the Salamon's complex is given by $$(-1)^{m}(m+1)\mathrm{td}(T^{c}M)[M].$$
\end{thm}

Note that such a manifold is a complex manifold and the number $\td(T^{c}M)[M]$ is the index of the Dolbeault complex associated to the complex tangent bundle $T^{c}M$ of $M$. In particular, it is an integer and so the above index is an integer divisible by $m+1$.

Let us remark here that by a different method very similar results were obtained in \cite{Bas} for certain class of quaternionic complexes. However, it is not clear to us whether the Salamon's complex was included.
\end{ex}
%%%%%%%%%%%%%%%%%%%%%%%%%%%%%%%%%%%%%%%%%%%%%%%%%%%%%%%%%%%%%%%%%%%%%%%%%%%%%%%%%%%%%%%%%%%%%%%%%%%%%%%%%%%%%%%%%%%%%%%%%%%%%%%%%%%%%%%%%%%%%%%%%%%%%%%%%%%%%%%%%%%%%%%%%%%%%%%%%%%%%%%%%%%%%%%%%%%%%%%%%%%%%%%%%%%%%%%%%%%%%%%%%%%%%%%%%%%%%%%%%%%%%%%%%%%%%%%%%%%%%%%%%%%%%%%%%%%%%%%%%%%%%%%%%%%%%%%%%%%%%%%%%%%%%%%%%%%%%%%%%%%%%%%%%%%%%%%%%%%%%%%%%%%%%%%%%%%%%%%%

\section{Quaternionic structures}
This section is devoted to basic topological properties of quaternionic structures, the main reference here is \cite{QS}. In particular, we define characteristic classes for these structures. Throughout the section, $X$ denotes a compact Hausdorff topological space.

\begin{defn}
Let $\beta$ be an oriented real $3$-dimensional vector bundle over $X$ with a positive-definite inner product $\langle-,-\rangle$. Then we define a \emph{bundle of quaternion algebras} as the vector bundle $\ha = \mathbb{R}\oplus\beta$ together with a fibrewise multiplication given by $$(s,u)\cdot(t,v) = (st-\langle u,v\rangle, sv+tu+u\times v).$$
Equivalently, if $\mathcal{P}\to X$ is the principal $\so(3)=\mathrm{Aut}(\mathbb{H})$-bundle corresponding to $\beta$, then we have $\ha=\mathcal{P}\times_{\mathrm{Aut}(\mathbb{H})} \mathbb{H}$. 
\end{defn}

The definition says that fibrewise the bundle $\ha$ carries a structure of the algebra of quaternions, but globally it may not be the product bundle $X\times\mathbb{H}$.

\begin{defn}
A real vector bundle $V\to X$ is said to be a \emph{right $\ha$-bundle} if it admits a right $\ha$-module structure, i.e. there is a bundle map $V\otimes_{\mathbb{R}}\ha \to V$ that restricts to an $\mathbb{H}$-module structure in each fibre.  
\end{defn}

It follows from the definition that the dimension of an $\ha$-bundle must be divisible by four. Moreover, such a bundle can be canonically oriented. Indeed, to orient a fibre $V_{x}$, choose a basis $e_1, e_2, \ldots, e_{m}$ of $V_{x}$ as an $(\ha)_{x}$-module and an oriented orthonormal basis $\mathrm{i, j, k}$ of $\beta_{x}$. Then $e_1, e_{1}\mathrm{i}, e_{1}\mathrm{j}, e_{1}\mathrm{k}, \ldots, e_{m}, e_{m}\mathrm{i}, e_{m}\mathrm{j}, e_{m}\mathrm{k}$ is the oriented basis of $V_{x}$.

\begin{prop}[\cite{QS}]\label{qstr}
A $4m$-dimensional real vector bundle $V$ is a right $\ha$-bundle for some oriented 3-dimensional vector bundle $\beta$ if and only if the structure group of the frame bundle of $V$ may be reduced to the subgroup $\mathrm{Sp}(1)\mathrm{Sp}(m)\subset \mathrm{GL}(4m,\mathbb{R})$.
\end{prop}
\begin{proof}
If the structure group of the frame bundle reduces to the subgroup $G=\Sp(1)\Sp(m)$, then there is a principal $G$-bundle $\mathcal{P}$ such that $V\cong \mathcal{P}\times_{G}\mathbb{H}^{m}$, where we view $\mathbb{H}^{m}$ as a real vector space. Now put $\beta=\mathcal{P}\times_{G}\mathrm{im}\,\mathbb{H}$ with the action of $G$ on $\mathrm{im}\,\mathbb{H}$ defined as follows: if $(a,A)\in \Sp(1)\times\Sp(m)$ represents an element of $G$, then $(a,A)\cdot q =aq\bar{a}$. Then $\beta$ is an orientable $3$-dimensional real vector bundle and the associated quaternion algebra is $\ha = \mathcal{P}\times_{G} \mathbb{H}$, where the action of $G$ on $\mathbb{H}$ is the same as on $\mathrm{im}\,\mathbb{H}$. But then right multiplication by quaternions is a $G$-equivariant map and so it induces a right $\ha$-module structure on $V$.

For the other direction see \cite{QS}. 
\end{proof}

The proposition applies, in particular, to the tangent or cotangent bundle of a quaternionic manifold $M$ (after introducing a Riemannian metric). We can actually describe the bundle $\beta$ as follows. Let $\mathbb{E}$ be the standard complex $\mathrm{Sp}(1)$-module as in Section \ref{1sec}. If we view the second symmetric power $S^{2}\mathbb{E}$ as a $G$-module, then the mapping $\varphi\colon\mathrm{im}\,\mathbb{H} \to S^{2}\mathbb{E}$ defined by $\varphi(u)= \mathrm{j}\otimes u - 1\otimes u\mathrm{j}$ is a real linear $G$-equivariant map. Moreover, the real basis $\mathrm{i}, \mathrm{j}, \mathrm{k}$ of $\mathrm{im}\,\mathbb{H}$ is mapped to a complex basis of $S^{2}\mathbb{E}$\begin{equation*} \mathrm{i}\mapsto (1\otimes \mathrm{j} + \mathrm{j}\otimes 1)\mathrm{i}, \quad \mathrm{j}\mapsto 1\otimes 1 +\mathrm{j}\otimes \mathrm{j}, \quad \mathrm{k}\mapsto (1\otimes 1- \mathrm{j}\otimes \mathrm{j})\mathrm{i}.\end{equation*} This implies that the complexification of $\mathrm{im}\,\mathbb{H}$ is isomorphic to $S^{2}\mathbb{E}$ and, on the level of associated vector bundles, the complexification of $\beta = \mathcal{P}\times_{G}\mathrm{im}\,\mathbb{H}$ is isomorphic to $S^{2}E = \mathcal{P}\times_{G}S^{2}\mathbb{E}$, which is a globally defined vector bundle over $M$. In fact, the sphere bundle of $\beta$ is precisely the Salamon's twistor space.\footnote{Compare with \cite[page 146]{SalK}.}

Now we proceed to define characteristic classes for $\ha$-bundles. Let $V\to X$ be a right $\ha$-bundle of quaternionic dimension $m$, i.e. real dimension $4m$. Then one can consider the associated projective bundle $\pi\colon \mathbb{H}_{\beta}\mathrm{P}(V)\to X$ whose fibre over a point $x\in X$ is the space of all quaternionic lines in the fibre $V_{x}$ in the sense of the $\ha$-module structure. Futhermore, let $L=\{(\ell,v) \in \mathbb{H}_{\beta}\mathrm{P}(V)\times V \,|\, v\in\ell\}$ be the canonical $\ha$-line bundle over $\mathbb{H}_{\beta}\mathrm{P}(V)$ oriented as a right $\pi^{*}\ha$-bundle. The following proposition defines characteristic classes $d^{\beta}_{j}(V)$ of the bundle $V$ as coefficients of a certain polynomial over the ring $H^{*}(X;\mathbb{Z})$.

\begin{prop}[\cite{QS}]
For each right $\ha$-bundle $V\to X$ of quaternionic dimension $m$ there are uniquely determined classes $d^{\beta}_{j}(V)\in H^{4j}(X;\mathbb{Z}), 1\leq j\leq m$, such that
$$H^{*}(\mathbb{H}_{\beta}\mathrm{P}(V);\mathbb{Z}) = H^{*}(X;\mathbb{Z})[t]/(t^{m}-d^{\beta}_1(V)t^{m-1} + \ldots + (-1)^{m}d^{\beta}_{m}(V)),$$ where $t=e(L)\in H^{4}(\mathbb{H}_{\beta}\mathrm{P}(V);\mathbb{Z})$ is the Euler class of the canonical bundle $L$.
\end{prop}
\begin{proof}
The proof is a standard application of the Leray-Hirsch theorem, see \cite{QS}. 
\end{proof}

One can verify (see \cite{QS}) that the classes $d^{\beta}_{j}(V)$ have usual properties of characteristic classes like naturality or multiplicativity. Moreover, there is a splitting principle which implies that in calculations with the classes $d^{\beta}_{j}(V)$ we may formally assume that there are $y_1, y_2, \ldots, y_{m}\in H^{4}(X;\mathbb{Z})$ such that $d^{\beta}_{j}(V)$ is the $j$-th elementary symmetric polynomial in the $y_{k}$'s or, in short, $d^{\beta}(V) = 1 + d^{\beta}_{1}(V) + \ldots + d^{\beta}_{m}(V) = \prod_{k=1}^{m}(1+y_{k})$.

The following proposition shows that the characteristic classes just defined determine other characteristic classes of $V$ as a real vector bundle. This is technically very useful.

\begin{prop}[\cite{QS}]\label{classes}
Let $V\to X$ be a canonically oriented right $\ha$-bundle of quaternionic dimension $m$.
\begin{enumerate}
\item[(a)] The Euler class $e(V)$ of $V$ equals the top-dimensional class $d^{\beta}_{m}(V)$.
\item[(b)] The rational Pontryagin classes $p_{j}(V)\in H^{4j}(X;\mathbb{Q})$ of $V$ are given by $$1+p_1(V)+p_2(V)+\ldots+p_{2m}(V) = \prod_{j=1}^{m}((1+y_{j})^2+p_1(\beta)),$$ where $d^{\beta}(V)= \prod_{k=1}^{m}(1+y_{k})$ and $p_{1}(\beta)$ is the first Pontryagin class of $\beta$.
\end{enumerate}
\end{prop}
\begin{proof}
See \cite{QS}, but note that we deal with right $\ha$-bundles rather than the left ones.
\end{proof}

Finally, the characteristic classes $d_{j}^{\beta}$ may be used to decribe the cohomology ring of the classifying space of the group $G=\Sp(1)\Sp(m)$. Let $EG\to BG$ be the universal principal $G$-bundle and put $\beta=EG\times_{G}\mathrm{im}\,\mathbb{H}$ and $\widetilde{V}=EG\times_{G}\mathbb{H}^{m}$. Then $\widetilde{V}$ is a right $\mathbb{H}_{\beta}$-bundle as in the proof of Proposition \ref{qstr}. Let us write simply $d_{j}$ for the characteristic classes $d^{\beta}_{j}(\widetilde{V})$ and $q_{1}$ for the first Pontryagin class $p_{1}(\beta)$ of $\beta$. 

\begin{prop}[\cite{QS}]
The rational cohomology ring of $B\mathrm{Sp}(1)\mathrm{Sp}(m)$ is given by
$$H^{*}(B\mathrm{Sp}(1)\mathrm{Sp}(m);\mathbb{Q}) \cong \mathbb{Q}[q_1, d_1, d_2, \ldots, d_{m}].$$
\end{prop}
\begin{proof}
One can obtain this from the description of the integral cohomology ring of the classifying space $B\mathrm{Sp}(1)\mathrm{Sp}(m)$, which was done in \cite{QS}.
\end{proof}

According to Proposition \ref{classes} the universal Euler class $e(\widetilde{V})$ equals the class $d_{m}$, which is a generator of the cohomology ring and so it is nonzero. We may therefore apply Proposition \ref{formula} to compute the indices of the quaternionic complexes.

%%%%%%%%%%%%%%%%%%%%%%%%%%%%%%%%%%%%%%%%%%%%%%%%%%%%%%%%%%%%%%%%%%%%%%%%%%%%%%%%%%%%%%%%%%%%%%%%%%%%%%%%%%%%%%%%%%%%%%%%%%%%%%%%%%%%%%%%%%%%%%%%%%%%%%%%%%%%%%%%%%%%%%%%%%%%%%%%%%%%%%%%%%%%%%%%%%%%%%%%%%%%%%%%%%%%%%%%%%%%%%%%%%%%%%%%%%%%%%%%%%%%%%%%%%%%%%%%%%%%%%%%%%%%%%%%%%%%%%%%%%%%%%%%%%%%%%%%%%%%%%%%%%%%%%%%%%%%%%%%%%%%%%%%%%%%%%%%%%%%%%%%%%%%%%%%%%%%%%%%

\section{The computations}

Having all the necessary background at hand, we finally describe an algorithm how to compute the indices of the elliptic complexes from Proposition \ref{quatcomplex}. This algorithm can be carried out for each given dimension of the manifold and for each given complex $D_{k}$.
 
Let $M$ be a compact $4m$-dimensional quaternionic manifold. By introducing a Riemannian metric on $M$ we may reduce the structure group of the principal frame bundle of $M$ to the subgroup $G=\mathrm{Sp}(1)\mathrm{Sp}(m)$. Then $M$ is canonically oriented, see the preceding section. Let $\mathcal{P}\to M$ be the principal $G$-bundle and $f\colon M \to BG$ its classifying map. Put $\widetilde{V}=EG\times_{G}\mathbb{R}^{4m}$ and $\widetilde{W}^{j}_{k} = EG\times_{G}\mathbb{W}^{j}_{k}$, where $W^{j}_{k}$ are the $G$-modules inducing the quaternionic complexes in question, see Proposition \ref{quatcomplex}. Then by the Atiyah-Singer formula from Proposition \ref{formula} the index of the quaternionic complex $D_{k}$ is given by \begin{equation}\label{final} \mathrm{ind}\,D_{k} = \left\{ f^{**}\left( \frac{\sum_{j=0}^{2m} (-1)^{j}\ch\widetilde{W}^{j}_{k}}{e(\widetilde{V})} \right) \cdot \td(TM\otimes\mathbb{C}) \right\}[M].\end{equation}

To evaluate this expression we first have to solve the equation \begin{equation}\label{eqn1} x\cup e(\widetilde{V}) = \sum_{j=0}^{2m} (-1)^{j}\ch \widetilde{W}^{j}_{k} \end{equation} in the cohomology ring $H^{**}(BG;\mathbb{Q})$. This task may be simplified in two ways. First, because the cohomology groups of the compact manifold $M$ vanish above dimension $4m$, it suffices to determine $x$ up to this dimension $4m$. Secondly, because the product group $G_{1}=\Sp(1)\times \Sp(m)$ is the double cover of $G$, the projection $\pi\colon G_1 \to G$ induces an isomorphism $(B\pi)^{**}\colon H^{**}(BG;\mathbb{Q}) \to H^{**}(BG_{1};\mathbb{Q})$. Therefore, we may pull back the above equation to $BG_{1}$ and solve it in $H^{**}(BG_{1};\mathbb{Q})$. The advantage will be clear soon.

Let $\mathbb{E}$ and $\mathbb{F}$ be the standard complex $\Sp(1)$ and $\Sp(m)$-modules, respectively, and put $\widetilde{E}=EG_{1}\times_{G_{1}}\mathbb{E}$ and $\widetilde{F}=EG_{1}\times_{G_{1}}\mathbb{F}$. Then these are globally defined vector bundles over $BG_{1}$ and we have \begin{equation} \label{split}(B\pi)^{*}(\widetilde{V}\otimes_{\mathbb{R}}\mathbb{C}) \cong \widetilde{E}\otimes_{\mathbb{C}}\widetilde{F}, \end{equation} compare with the isomorphism \eqref{EF}. Moreover, if $\beta=EG\times_{G}\mathrm{im}\,\mathbb{H}$, then the real vector bundle $\widetilde{V}_{1} = (B\pi)^{*}(\widetilde{V})$ is a right $\mathbb{H}_{\beta_1}$-bundle for $\beta_{1}=(B\pi)^{*}(\beta)$. As in the previous section, we may prove that $\beta_{1}\otimes\mathbb{C} \cong S^{2}\widetilde{E}$ and hence for the first Pontryagin class of $\beta_{1}$ we get \begin{equation} \label{pontr} p_{1}(\beta_{1}) = -c_{2}(S^{2}\widetilde{E}) = -4c_{2}(\widetilde{E}).\end{equation} Applying the Chern character on \eqref{split} and comparing inductively the two sides of the result, one may write the Pontryagin classes of $\widetilde{V}_{1}$ as polynomials in the Chern classes of $\widetilde{E}$ and $\widetilde{F}$. Altogether with \eqref{pontr} and Proposition \ref{classes} this implies that we are able to translate between three sets of characteristic classes -- the Chern classes of $\widetilde{E}$ and $\widetilde{F}$, the Pontryagin classes of $\widetilde{V}_{1}$ and $\beta_{1}$ and, finally, the classes $d^{\beta_{1}}_{1}(\widetilde{V}_{1}), d^{\beta_{1}}_{2}(\widetilde{V}_{1}), \ldots, d^{\beta_{1}}_{m}(\widetilde{V}_{1})$ and $p_{1}(\beta_{1})$.

Now we may return to the equation \eqref{eqn1} in the pulled-back version \begin{equation} \label{eqn2} (B\pi)^{*}(x)\cup e(\widetilde{V}_{1}) = \sum_{j=0}^{2m} (-1)^{j} \ch ((B\pi)^{*}(\widetilde{W}^{j}_{k})). \end{equation} We would like to compute the right-hand side in terms of the Chern classes of $\widetilde{E}$ and $\widetilde{F}$. Once we do this, it remains to express the result in terms of the $d^{\beta_{1}}_{l}$-classes and $p_{1}(\beta_{1})$ and divide by $e(\widetilde{V}_{1})=d^{\beta_{1}}_{m}(\widetilde{V}_{1})$ to obtain the solution $(B\pi)^{*}(x)\in H^{*}(BG_{1}; \mathbb{Q})$ and hence also $x\in H^{*}(BG;\mathbb{Q})$, i.e. the fraction in \eqref{final}.

Recall from the definition of the representations $\mathbb{W}^{j}_{k}$  (see \eqref{w}) that $$ (B\pi)^{*}(\widetilde{W}_{k}^{j}) =  S^{j+k}\widetilde{E}\otimes (\Lambda^{j}\widetilde{F}\otimes S^{k}\widetilde{F}^{*})_{0} \quad \text{for } j<2m, \quad (B\pi)^{*}(\widetilde{W}_{k}^{2m}) = S^{2(m+k)}\widetilde{E}\otimes \Lambda^{2m}\widetilde{F}.$$ The two factors in the tensor products are globally defined vector bundles and so we can compute its Chern characters separately. We will again use the approach described in Remark \ref{borhirz}. As a maximal torus $S$ of $G_{1}=\Sp(1)\times \Sp(m)$ take the direct product of the standard maximal tori of $\Sp(1)$ and $\Sp(m)$ -- the standard maximal torus of $\Sp(1)$ is the set of complex units $\exp(2\pi\mathrm{i}x)$ while the standard maximal torus of $\Sp(m)$ is the set of diagonal matrices with entries $\exp(2\pi\mathrm{i}x_{l})$, where $x_{l}\in\mathbb{R}$.

Consider first the vector bundle $\widetilde{E}=EG_{1}\times_{G_{1}} \mathbb{E}$. The weights of the corresponding $G_{1}$-module $\mathbb{E}$ are $\pm x$ viewed as linear forms on the Lie algebra $\mathfrak{s}$. Then the total Chern class of $\widetilde{E}$ is given by $c(\widetilde{E}) = (1+y)(1-y)$ and $$c_{1}(\widetilde{E}) = 0,\quad c_{2}(\widetilde{E})= -y^{2}.$$

The Chern classes of the symmetric powers $S^{j}\widetilde{E}$ are now easy to compute. Clearly, the weights of the $G_{1}$-module $S^{j}\mathbb{E}$ are the forms $(k_{1}-k_{2})x$, $k_1+k_2=j$, and so we have $$c(S^{j}\widetilde{E}) = \prod_{k_1+k_2=j}(1+(k_1-k_2)y).$$ This is clearly a polynomial expression in $-y^{2}=c_{2}(\widetilde{E})$.

Now turn to the vector bundle $\widetilde{F}=EG_{1}\times_{G_{1}}\mathbb{F}$. The weights of the $G_{1}$-module $\mathbb{F}$ are precisely $\pm x_{l}$, $1\leq l \leq m$, viewed as linear forms on the Lie algebra $\mathfrak{s}$. The total Chern class of $\widetilde{F}$ is then given by $$c(\widetilde{F}) = \prod_{l=1}^{m}(1+y_{l})(1-y_{l})=\prod_{l=1}^{m}(1-y_{l}^{2})$$ and so $c_{2j}(\widetilde{F})$ is the $j$-th elementary symmetric polynomial in the $-y_{l}^{2}$ while $c_{2j+1}(\widetilde{F})=0$.

The computation of the Chern classes of $(\Lambda^{j}\widetilde{F}\otimes S^{k}\widetilde{F}^{*})_{0}$, $k\geq 0$, is a bit more complicated. This is because $\mathbb{V}^{j}_{k}=(\Lambda^{j}\mathbb{F}\otimes S^{k}\mathbb{F}^{*})_{0}$ was defined as a representation of the group $\mathrm{U}(2m)$ corresponding to some maximal weight and we have to find its weights with respect to the subgroup $\Sp(m)\subset \mathrm{U}(2m)$. This can be achieved as follows. First, the character ring of complex representations of the group $\mathrm{SU}(2m)$ differs from that of the group $\mathrm{U}(2m)$ only by a one-dimensional determinantal representation on which $\Sp(m)$ acts trivially. Therefore, there is nothing lost in assuming that $\mathbb{V}^{j}_{k}$ is a representation of $\mathrm{SU}(2m)$. But $\mathrm{SU}(2m)$ is compact and simply connected and so its representation theory is equivalent to that of the complex Lie algebra $\mathfrak{sl}(2m, \mathbb{C})$. In particular, if we know the maximal weight of $\mathbb{V}^{j}_{k}$, the remaining weights can be computed by standard algorithms, see for example \cite{Sam}. The maximal weight of $\mathbb{V}^{j}_{k}$ is by definition the sum of the maximal weight of $\Lambda^{j}\mathbb{F}$ and the maximal weight of $S^{k}\mathbb{F}^{*}$ and these are easy to find -- if $z_{1}, z_{2}, \ldots, z_{2m}$ are the weights of $\mathbb{F}$ viewed as the standard $\mathrm{SU}(2m)$-module, then the maximal weight of $\Lambda^{j}\mathbb{F}$ is $z_{1}+z_{2}+\ldots+z_{j}$ while the maximal weight of $S^{k}\mathbb{F}^{*}$ is $-k\cdot z_{2m}$. The remaining weights of $\mathbb{V}^{j}_{k}$ are  integral linear combinations of the $z_{l}$'s as well. To obtain the weights of $\mathbb{V}^{j}_{k}$ as a $\Sp(m)$-module we only have to substitute $z_{2l}=x_{l}$ and $z_{2l+1}=-x_{l}$ for $1 \leq l \leq m$ -- this can be seen from the definition of the standard inclusion $\Sp(m)\subset \mathrm{SU}(2m)$. Finally, once we know the weights, we get the total Chern class $c(\mathbb{V}^{j}_{k})$ and this will be a polynomial expression symmetric in the variables $-y_{l}^{2}$. Indeed, the set of weights of a $\mathrm{SU}(2m)$-module is invariant under the action of the Weyl group of $\mathrm{SU}(2m)$, which is the symmetry group on the set $\{z_1, z_2, \ldots, z_{2m} \}$. Therefore, $c(\mathbb{V}^{j}_{k})$ can be expressed in terms of the Chern classes of $\widetilde{F}$.

To sum up, we have seen that the Chern classes of both the vector bundles $S^{j+k}\widetilde{E}$ and $(\Lambda^{j}\widetilde{F}\otimes S^{k}\widetilde{F}^{*})_{0}$ may be written in terms of the Chern classes of $\widetilde{E}$ and $\widetilde{F}$ and so the same holds true for the Chern character of $(B\pi)^{*}(\widetilde{W}^{j}_{k})$. The right-hand side of the equation \eqref{eqn2} is thus a polynomial in the Chern classes of $\widetilde{E}$ and $\widetilde{F}$ and so it can be expressed in terms of the classes $d^{\beta_{1}}_{l}(\widetilde{V}_{1})$ and $p_{1}(\beta_{1})$. Next, the result will be a multiple of the Euler class $e(\widetilde{V}_{1}) =
d^{\beta_{1}}_{m}(\widetilde{V}_{1})$ and by dividing we obtain $(B\pi)^{*}(x)\in H^{*}(BG_{1}; \mathbb{Q})$. To get the solution $x\in H^{*}(BG; \mathbb{Q})$ of \eqref{eqn1} it suffices to write $d^{\beta}_{l}(\widetilde{V})$ and $p_{1}(\beta)$ instead of $d^{\beta_{1}}_{l}(\widetilde{V}_{1})$ and $p_{1}(\beta_{1})$.

We are now at the end of the algorithm. The solution $x$, which is the fraction in \eqref{final}, may be expressed in terms of the Pontryagin classes of $\widetilde{V}$ and $p_{1}(\beta)$ so that $f^{**}(x)$ will be a polynomial in the Pontryagin classes of $TM$ and the class $p_{1}(f^{*}\beta)$. By multiplying with the Todd class $\td(TM\otimes\mathbb{C})$ and evaluating the top-dimensional part of the product on the fundamental class $[M]$ we get the desired index.

Index formulas obtained in this way depend on the $G$-structure of $M$ via the characteristic class $p_{1}(f^{*}\beta)$. This class may be expressed without any reference to the classifying map $f$. Indeed, from the isomorphism $\beta\otimes \mathbb{C}\cong S^{2}\widetilde{E}$, where $S^{2}\widetilde{E}$ is now viewed as a vector bundle over $BG$, follows that $f^{*}(\beta\otimes\mathbb{C}) \cong S^{2}E$ and this is a globally defined complex vector bundle over $M$. Hence $$p_{1}(f^{*}\beta) = -c_{2}(f^{*}(\beta\otimes\mathbb{C})) = -c_{2}(S^{2}E).$$

In general, the most difficult computational problem is to find the weights of the $G$-modules $(\Lambda^{j}\mathbb{F}\otimes S^{k}\mathbb{F}^{*})_{0}$ and then process these to obtain the Chern classes of the vector bundles $(\Lambda^{j}\widetilde{F}\otimes S^{k}\widetilde{F}^{*})_{0}$. Of course, one can make use of computer algebra systems such as LiE (see \cite{Lie}) and Maple. We have carried out some calculations for 8 and 12-dimensional manifolds arriving at the following formulas.

\begin{thm}\label{eight}
Let $M$ be an $8$-dimensional compact quaternionic manifold. If we denote $p_{1}=p_{1}(TM)$, $p_{2}=p_{2}(TM)$ and  $q_{1}=-c_{2}(S^{2}E)$, then we have \begin{align*} \textup{ind}\,D_{0} = \left(\frac{7}{1920}p_{1}^{2} -\frac{1}{24}p_{1}q_{1} -\frac{1}{480}p_{2} +\frac{1}{12}q_{1}^{2} \right)[M], \\  \textup{ind}\,D_{1} = \left(\frac{209}{1920}p_{1}^{2} +\frac{11}{24}p_{1}q_{1} -\frac{167}{480}p_{2} +\frac{25}{12}q_{1}^{2} \right)[M].\end{align*} 
\end{thm}

\begin{thm}
Let $M$ be a $12$-dimensional compact quaternionic manifold. If we denote $p_{1}=p_{1}(TM)$, $p_{2}=p_{2}(TM)$, $p_{3}=p_{3}(TM)$ and $q_{1}=-c_{2}(S^{2}E)$, then we have \begin{align*} \textup{ind}\,D_{0} = \biggl( \frac{31}{241920}p_{1}^{3} -\frac{7}{2304}p_{1}^{2}q_{1} -\frac{11}{60480}p_{1}p_{2} &+\frac{41}{2304}p_{1}q_{1}^{2} + \\ &+\frac{1}{576}p_{2}q_{1} +\frac{1}{15120}p_{3} -\frac{73}{2304}q_{1}^{3} \biggr)[M], \end{align*} \begin{align*}\quad \textup{ind}\,D_{1} = \biggl( -\frac{1}{6720}p_{1}^{3} -\frac{77}{576}p_{1}^{2}q_{1} +\frac{1}{280}p_{1}p_{2} -\frac{35}{576}p_{1}q_{1}^{2} +\frac{7}{18}p_{2}q_{1} -\frac{17}{840}p_{3} -\frac{623}{576}q_{1}^{3} \biggr)[M]. \end{align*} 
\end{thm}

Recall that the analytical index is an integer and so the above formulas evaluated on the fundamental class must give an integer as well. We will verify this for the quaternionic projective spaces.

\begin{ex} Let $M=\mathbb{H}\mathrm{P}^{m}$. In this case the vector bundles $E$ and $F$ exist globally and $E$ is precisely the tautological line bundle. The cohomology ring $H^{*}(\mathbb{H}\mathrm{P}^{m};\mathbb{Z})$ is generated by the class $u=-c_2(E)$, which also satisfies $(u^{m})\,[\mathbb{H}\mathrm{P}^{m}]=1$. Moreover, one can show (see \cite{BH}) that the Pontryagin classes of $T\mathbb{H}\mathrm{P}^{m}$ are given by $$p(T\mathbb{H}\mathrm{P}^{m}) = (1+u)^{2m+2}(1+4u)^{-1},$$ where $(1+4u)^{-1}$ is the inverse formal power series to $1+4u$. Finally, by \eqref{pontr} we have $q_{1} = -4c_{2}(E) = 4u$.

Take first $m=2$. Then the Pontryagin classes are $p_1(T\mathbb{H}\mathrm{P}^{2})=2u$ and $p_2(T\mathbb{H}\mathrm{P}^{2})=7u^2$ and inserting into the formulas from Theorem \ref{eight} we obtain \begin{align*} \mathrm{ind}\,D_{0} =  1, \quad \mathrm{ind}\,D_{1} = 35.\end{align*}
Similarly, for $m=3$ we have $p_{1}(T\mathbb{H}\mathrm{P}^{3})=4u$, $p_{2}(T\mathbb{H}\mathrm{P}^{3})=12u^{2}$, $p_{3}(T\mathbb{H}\mathrm{P}^{3})=8u^{3}$ and \begin{align*} \mathrm{ind}\,D_{0} =  -1, \quad \textup{ind}\,D_{1} = -63. \end{align*}
\end{ex}

A drawback of the index formulas is that they depend on the class $q_{1}$, which is not easy to compute. However, by taking integral linear combinations we may try to eliminate the terms containing $q_{1}$ and thus obtain some integrality conditions on the Pontryagin classes of a quaternionic manifold. Consider for example the two formulas from Theorem \ref{eight} and the following linear combinations \begin{align*} 11\cdot \mathrm{ind}\,D_{0}+ \mathrm{ind}\,D_{1} &= \biggl(\frac{143}{960}p_{1}^{2}-\frac{89}{240}p_{2}+3q_{1}^{2}\biggr)[M], \\ 50\cdot \mathrm{ind}\,D_{0}- 2\cdot\mathrm{ind}\,D_{1} &= \biggl( -\frac{17}{480}p_{1}^{2}-3p_{1}q_{1}+\frac{71}{120}p_{2}\biggr)[M].\end{align*} But $p_{1}$ and $q_{1}$ are Pontryagin classes of some vector bundles and so they come from integral cohomology. In particular, evaluation of $p_{1}q_{1}$ and $q_{1}^{2}$ on the fundamental class of $M$ gives an integer. But then by evaluating the rest of the above formulas we must again obtain an integer (and not only a rational number).

\begin{cor}\label{cond}
Let $M$ be an 8-dimensional compact quaternionic manifold. Then the following expressions are integers $$ \biggl(\frac{143}{960}p_{1}(TM)^{2}-\frac{89}{240}p_{2}(TM)\biggr)[M], \quad \biggl( -\frac{17}{480}p_{1}(TM)^{2}+\frac{71}{120}p_{2}(TM)\biggr)[M].$$
\end{cor}  

Of course we may deal with other integral linear combinations $a\cdot \mathrm{ind}\,D_{0}+b\cdot \mathrm{ind}\,D_{1}$ to derive more integrality conditions and similarly for the $12$-dimensional manifolds. 

As a final point, let us remark that for manifolds admitting a $\mathrm{GL}(m,\mathbb{H})$-structure with a torsion-free connection (see also Example \ref{hyper}) the formulas simplify considerably because the vector bundle $E$ is trivial and hence $q_{1}=0$. Moreover, the vector bundle $F$ is isomorphic to the complex tangent $T^{c}M$ of $M$. Assuming $m=2$ we then easily compute that $p_{1}(TM)=-2c_{2}(F)$ and $p_{2}(TM)=c_{2}(F)^{2}+2c_{4}(E)$ and substituting to the formula from Theorem \ref{eight} for the Salamon's complex $D_{0}$ we obtain \begin{align*} \mathrm{ind}\,D_{0} = \left(\frac{1}{80}c_{2}(F)^{2}-\frac{1}{240}c_{4}(F)\right)[M].\end{align*}
But this is exactly three times the top-dimensional part of $\td(F)=\td(T^{c}M)$, which verifies that the general formula for the Salamon's complex from Theorem \ref{eight} coincides with that in Theorem \ref{hyperindex}.

\section*{Acknowledgements}
The paper presents the results of my diploma thesis and in this connection thanks go to Martin \v{C}adek, my supervisor. Furthermore, I am grateful to Andreas \v{C}ap, Vladim\'{i}r Sou\v{c}ek and Petr Somberg for several discussions and to Luk\'{a}\v{s} Vok\v{r}\'{i}nek for his remarks and comments. The research was supported by the grant MSM0021622409 of the Czech Ministry of Education.


\begin{thebibliography}{99}
\bibitem{AS3} M.\,F.\,Atiyah, I.\,M.\,Singer, \emph{The index of elliptic operators: III}, Ann. of Math. \textbf{87} (1968), 546-604.
\bibitem{Bas} R.\,J.\,Baston, \emph{Quaternionic complexes}, J. Geom. Phys. \textbf{8} (1992), 29-52.
\bibitem{Bor} A.\,Borel, \emph{Sur la cohomologie des espaces fibr\'{e}s principaux et des espaces homog\`{e}nes de groupes de Lie compacts}, Ann. of Math. \textbf{57} (1953), 115-207.
\bibitem{BH} A.\,Borel, F.\,Hirzebruch, \emph{Characteristic classes and homogeneous spaces, I}, American J. of Math. \textbf{80} (1958), 458-538.
\bibitem{QS} M.\,\v{C}adek, M.\,C.\,Crabb, J.\,Van\v{z}ura, \emph{Quaternionic structures}, in preparation.
\bibitem{DEF} A.\,\v{C}ap, \emph{Infinitesimal automorphisms and deformations of parabolic geometries}, J. Eur. Math. Soc. \textbf{10, 2} (2008), 415-437.
\bibitem{PAR} A.\,\v{C}ap, J.\,Slov\'{a}k, \emph{Parabolic geometries I: Background and general theory}, AMS, Providence, 2009.
\bibitem{BGG} A.\,\v{C}ap, J.\,Slov\'{a}k, V.\,Sou\v{c}ek, \emph{Bernstein-Gelfand-Gelfand sequences}, Ann. of Math. \textbf{154} (2001), 97-113, an extended version electronically available at \verb(http://www.esi.ac.at(.
\bibitem{SUB} A.\,\v{C}ap, V.\,Sou\v{c}ek, \emph{Subcomplexes in curved BGG-sequences}, to appear, ESI preprint available at \verb(http://www.esi.ac.at(.
\bibitem{LBS} C.\,LeBrun, S.\,Salamon, \emph{Strong rigidity of positive quaternion K\"{a}hler manifolds}, Inventiones mathematicae \textbf{118} (1994), 109-132.
\bibitem{Sal} S.\,Salamon, \emph{Differential geometry of quaternionic manifolds}, Annales scientifiques de l'\'E.\,N.\,S. \textbf{19} (1986), 31-55. 
\bibitem{SalK} S.\,Salamon, \emph{Quaternionic K\"{a}hler manifolds}, Inventiones mathematicae \textbf{67} (1982), 143-171.
\bibitem{Sam} H.\,Samelson, \emph{Notes on Lie algebras}, third ed., Springer-Verlag, New York, 1990.
\bibitem{Lie} Computer algebra system LiE, available at \verb(http://young.sp2mi.univ-poitiers.fr/~marc/LiE/(.
\end{thebibliography}
\end{document}